\documentclass[12pt]{amsart}
\usepackage{amscd,amsmath,amsthm,amssymb}
\usepackage{tikz}

 \usepackage{multicol}
 
\tikzstyle{punkt}=[circle, fill=black, minimum size=1mm,inner sep=0pt, draw]

\usepackage{amsfonts,amssymb,amscd,amsmath,enumerate,verbatim}

\usetikzlibrary{arrows}
\input xy
\xyoption{all}
\def\frk{\mathfrak}               

\def\Phi{{\frk N}}
\def\opn#1#2{\def#1{\operatorname{#2}}} 
\opn\chara{char} \opn\length{\ell} \opn\pd{pd} \opn\rk{rk}
\opn\projdim{proj\,dim} \opn\injdim{inj\,dim} \opn\rank{rank}
\opn\depth{depth} \opn\grade{grade} \opn\height{height}
\opn\embdim{emb\,dim} \opn\codim{codim}

\opn\Tr{Tr} \opn\bigrank{big\,rank}
\opn\superheight{superheight}\opn\lcm{lcm}
\opn\trdeg{tr\,deg}
\opn\reg{reg} \opn\lreg{lreg} \opn\ini{in} \opn\lpd{lpd}
\opn\size{size}\opn{\mult}{mult}
\opn\div{div} \opn\Div{Div} \opn\cl{cl} \opn\Cl{Cl}
\opn\Spec{Spec} \opn\Supp{Supp} \opn\supp{supp} \opn\Sing{Sing}
\opn\Ass{Ass} \opn\Min{Min}
\opn\Ann{Ann} \opn\Rad{Rad} \opn\Soc{Soc}
\opn\Syz{Syz} \opn\Im{Im} \opn\Ker{Ker} \opn\Coker{Coker}
\opn\Am{Am} \opn\Hom{Hom} \opn\Tor{Tor} \opn\Ext{Ext}
\opn\End{End} \opn\Aut{Aut} \opn\id{id} \opn\ini{in}

\opn\nat{nat}
\opn\pff{pf}
\opn\Pf{Pf} \opn\GL{GL} \opn\SL{SL} \opn\mod{mod} \opn\ord{ord}
\opn\Gin{Gin}
\opn\Hilb{Hilb}\opn\adeg{adeg}\opn\std{std}\opn\ip{infpt}
\opn\Pol{Pol}
\opn\sat{sat}
\opn\Var{Var}
\opn\Gen{Gen}

\opn\aff{aff} \opn\con{conv} \opn\relint{relint} \opn\st{st}
\opn\lk{lk} \opn\cn{cn} \opn\core{core} \opn\vol{vol}
\opn\link{link} \opn\star{star}
\opn\gr{gr}

\def\pot#1#2{#1[\kern-0.28ex[#2]\kern-0.28ex]}

\opn\dirlim{\underrightarrow{\lim}}
\opn\inivlim{\underleftarrow{\lim}}
\def\Implies{\ifmmode\Longrightarrow \else
        \unskip${}\Longrightarrow{}$\ignorespaces\fi}
\def\implies{\ifmmode\Rightarrow \else
        \unskip${}\Rightarrow{}$\ignorespaces\fi}
\def\iff{\ifmmode\Longleftrightarrow \else
        \unskip${}\Longleftrightarrow{}$\ignorespaces\fi}

\let\:=\colon
\newtheorem{Theorem}{Theorem}[section]
\newtheorem{Lemma}[Theorem]{Lemma}
\newtheorem{Corollary}[Theorem]{Corollary}

\newtheorem{Example}[Theorem]{Example}

\newtheorem{Question}[Theorem]{Question}
\let\epsilon\varepsilon
\let\phi=\varphi
\let\kappa=\varkappa
\def\qed{\ifhmode\textqed\fi
      \ifmmode\ifinner\quad\qedsymbol\else\dispqed\fi\fi}
\def\textqed{\unskip\nobreak\penalty50
       \hskip2em\hbox{}\nobreak\hfil\qedsymbol
       \parfillskip=0pt \finalhyphendemerits=0}
\def\dispqed{\rlap{\qquad\qedsymbol}}

\opn\dist{dist}
\def\pnt{{\raise0.5mm\hbox{\large\bf.}}}

\opn\Lex{Lex}
\opn\diam{diam}

\begin{document}
	
\title{Depth and Krull dimension of Binomial edge ideals}
\author[T.~Hibi]{Takayuki Hibi}
\address[Takayuki Hibi]
{Department of Pure and Applied Mathematics, 
	Graduate School of Information Science and Technology, 
	The University of Osaka, 
	Suita, Osaka 565-0871, Japan}
\email{hibi@math.sci.osaka-u.ac.jp}
\author[S.~Saeedi~Madani]{Sara Saeedi Madani}
\address[Sara Saeedi Madani]
{Department of Mathematics and Computer Science, Amirkabir University of Technology (Tehran Polytechnic), Tehran, Iran, and School of Mathematics, Institute for Research in Fundamental Sciences (IPM), P.O. Box: 19395-5746, Tehran, Iran} 
\email{sarasaeedi@aut.ac.ir, sarasaeedim@gmail.com}
\subjclass[2020]{05E40, 13C15}
\keywords{Binomial edge ideal, depth, Krull dimension}
\thanks{
The research of the second author was in part supported by a grant from IPM (No. 1405130118).
}

\begin{abstract}
 Let $J_G$ denote the binomial edge ideal of a finite graph $G$ in the polynomial ring $S$. We determine all triples $(n,t,d)$ with $n\geq 3$ for which there exists a finite connected graph $G$ on $n$ vertices with $\depth(S/J_G)=t$ and $\dim(S/J_G)=d$.    	
\end{abstract}

\maketitle

\section*{Introduction}
Every graph is a finite graph having no loop, no multiple edge and no isolated vertex.  Let $G$ be a graph on the vertex set $V(G)=[n]=\{1,\ldots,n\}$ and $E(G)$ the set of edges of $G$.  An \emph{induced subgraph} of $G$ on $W\subseteq [n]$ is a subgraph $G|_{W}$ whose vertex set is $W$ and its edge set consists of those edges $\{i,j\} \in E(G)$ with $\{i,j\} \subseteq W$.  A vertex $i$ of $G$ is called a \emph{cut point} of $G$ if $G|_{[n]-\{i\}}$ has more connected components than $G$.

Let $S=K[x_1, \ldots, x_n, y_1, \ldots, y_n]$ denote the polynomial ring in $2n$ variables over a field $K$.  The {\em binomial edge ideal}, introduced by \cite{HHHKR} and \cite{O} independently, of $G$ is the ideal $J_G$ of $S$ generated by the binomials $x_iy_j - x_jy_i$ with $i < j$ and $\{i,j\} \in E(G)$.     

Given $T\subset [n]$, a prime ideal $P_T(G)$ associated with $T$ is defined to be 
\[
P_T(G)=(x_i,y_i: i\in T)+J_{\Tilde{G_1}}+\cdots+J_{\Tilde{G}_{c_G(T)}},
\]
where $G_1,\ldots,G_{c_G(T)}$ are the connected components of $G|_{[n]-T}$ and where $\Tilde{H}$ is the complete graph on the vertex set of a graph $H$.  We say that $T\subset [n]$ has the \emph{cut point property} if each $i\in T$ is a cut point of the induced graph $G|_{([n]-T)\cup \{i\}}$ of $G$. Let 
\[
\mathcal{C}(G):=\{T\subset [n]: T=\emptyset~\textit{or}~T~\text{has the cut point property}\}.
\]
As was shown in \cite{HHHKR}, the minimal prime ideals of $J_G$ are described as follows:
\[ 
\Min(J_G)=\{P_T(G): T\in \mathcal{C}(G)\}. 
\]

It is known \cite{HHHKR} that 
\begin{equation}
\label{dimension formula}
\dim(S/J_G)=\max \{n-|T|+c_G(T): T\in \mathcal{C}(G)\}.
\end{equation}
In particular, if $G$ is connected and $n\geq 3$, then one has
\[
n+1\leq \dim(S/J_G) \leq 2n-2.
\]

A {\em clique} of $G$ is a complete subgraph of $G$.  A {\em free} vertex of $G$ is a vertex of $G$ which is contained in exactly one maximal clique of $G$.  Let $f(G)$ denote the number of free vertices of $G$.  The {\em length} of a path is the number of edges of the path.  

Let $G$ be a connected graph. The \emph{distance} $\dist_G(i,j)$ of two distinct vertices $i$ and $j$ in $[n]$ is the smallest length of paths connecting $i$ and $j$ in $G$. The \emph{diameter} of $G$ is
$\d(G) = \max\{\dist_G(i,j):i,j \in[n]\}$.
Let $\kappa(G)$ denote the \emph{vertex connectivity} of $G$, i.e. the minimum number of vertices of $G$ whose deletion disconnects $G$.

It is known \cite{BN, RSK2} that if $G$ is non--complete, then one has
\begin{equation}
\label{depth}
f(G) + d(G) \leq \depth (S/J_G) \leq n - \kappa(G) + 2.
\end{equation}
In particular, together with \cite[Theorem~1.1]{EHH}, for any finite connected graph $G$ one has
\[
\depth(S/J_G) \leq n+1.
\]
Furthermore, it is shown \cite[Theorem 5.2]{RSK1} that if $n \geq 3$, then
\[
\depth(S/J_G) \geq 4.
\] 

The purpose of the present paper is to give an answer to the question as follows:  

\begin{Question}[\cite{HS-CW}]
\label{Question}   
 Given integers $t,n,d$ with $4\leq t\leq n+1\leq d\leq 2n-2$, does there exist a finite connected graph $G$ on $n$ vertices for which $\depth(S/J_G)=t$ and $\dim(S/J_G)=d$\,?   
\end{Question}

In Section~\ref{maximum depth section}, we mainly focus on $t=n+1$ which is the maximum possible value for $t$ and show that the answer to Question~\ref{Question} is positive for triples $(n,n+1,d)$ with $n \geq 3$ and $n+1\leq d\leq 2n-2$ (Theorem~\ref{maximum depth}). Moreover, we show that triples $(n,t,2n-2)$ for $n\geq 3$ and $(n,t,2n-3)$ for $n\geq 5$ provide a positive answer to Question~\ref{Question} if and only if $t=n+1$, namely the maximum possible value of $t$ (Corollary~\ref{higher d}). In Section~\ref{small depth section}, we determine all triples $(n,t,d)$ for the remaining cases of $d$ with $n\geq 4$ and $t=4$ as well as $n\geq 5$ and $t=5$, namely two smallest values of $t$, for which Question~\ref{Question} has a positive answer (Example~\ref{n=4}, Theorem~\ref{t=4} and Theorem~\ref{t=5}). Finally, in Section~\ref{other values section}, we give a positive answer to Question~\ref{Question} for all the remaining triples $(n,t,d)$, namely with $6\leq t\leq n$ and $n+1\leq d\leq 2n-4$ (Corollary~\ref{remaining cases}).

\section{The maximum value of $t$ in Question~\ref{Question}}\label{maximum depth section} 

A \emph{chordal} graph is a graph each of whose cycles of lenght $>3$ has a chord.  A \emph{block} graph is a chordal graph for which any two distinct maximal cliques intersect in at most one vertex. In the sequel, we use the following result from \cite{EHH}. 

\begin{Theorem}
	[\cite{EHH}]
    \label{block}
	Let $G$ be a connected block graph with $n$ vertices. Then one has $$\depth(S/J_G)=n+1.$$    
\end{Theorem}

Our discussion starts with the maximum possible value for the depth in the next theorem.  Let $K_n$ denote the complete graph on $n$ vertices. A vertex is called a {\em leaf} if it is contained in exactly one edge.  

\begin{Theorem}
    \label{maximum depth}
    Let $n,d$ be integers with $n\geq 3$ and $n+1 \leq d \leq 2n-2$. Then there exists a finite connected graph on $n$ vertices with $\depth(S/J_G)=n+1$ and $\dim(S/J_G)=d$. 
\end{Theorem}

\begin{proof}
Let $d=n+1$.  It follows from Theorem~\ref{block} and $\mathcal{C}(G)=\{\emptyset\}$ that $G=K_n$ satisfies
$\depth(S/J_G)=\dim(S/J_G)=n+1$. 
 
Let $d=n+\ell$, where $2\leq \ell\leq n-2$. Define $G$ to be the graph obtained from the complete graph $K_{n-\ell}$ by attaching $\ell$ leaves to one of the vertices of $K_{n-\ell}$, say~$i$. Since $G$ is a connected block graph, it follows from Theorem~\ref{block} that $\depth(S/J_G)=n+1$.  Since $\mathcal{C}(G)=\{\emptyset, \{i\}\}$. 
It follows from \eqref{dimension formula} that 
 \[
 \dim(S/J_G)=\max\{n+1, n+\ell\}=n+\ell,
 \]    
 since $\ell\geq 2$. This completes the proof.
 \hspace{7cm}
\end{proof}

In the next two lemmata, we focus on two higher values of $d$, namely $2n-3$ and $2n-2$.  Let $K_{1,s-1}$ denote the star graph on $[s]$ with the edges $\{1,2\},\{1,3\},\ldots, \{1,s\}$.  

\begin{Lemma}
	\label{d=2n-2}
	Let $n,t$ be integers with $n\geq 4$ and $t=4,\ldots,n$. Then there is no finite connected graph on $n$ vertices with $\depth(S/J_G)=t$ and $\dim(S/J_G)=2n-2$. 
\end{Lemma}

\begin{proof}
  Suppose on the contrary that there exists a finite connected graph $G$ on $[n]$ with $\depth(S/J_G)=t$ and $\dim(S/J_G)=2n-2$. Since $\dim(S/J_G)=2n-2$, it follows from \eqref{dimension formula} that there is a subset $T$ of $[n]$ for which $n-|T|+c_G(T)=2n-2$, and hence $n+|T|=c_G(T)+2$. Since $n\geq 4$, one has $2n-2>n+1$ which implies that $|T|\geq 1$. Since $c_G(T)\leq n-1$, it follows that 
  $$c_G(T)+2\leq n+1\leq n+|T|=c_G(T)+2,$$ and hence $c_G(T)=n-1$. Therefore, $|T|=1$ and $G$ is the star graph $K_{1,n-1}$ which is a block graph. One has $\depth(S/J_G)=n+1$ (Theorem~\ref{block}), a contradiction.      	
\end{proof}

If $i\in [n]$, then for simplicity we denote the induced subgraph of $G$ on $[n]-\{i\}$, by $G-i$.  The {\em join} $G*G'$ of two graphs $G$ and $G'$ with $V(G) \cap V(G') = \emptyset$ is the graph on $V(G)\cup V(G')$ with
\[
E(G*G')=E(G)\cup E(G')\cup \{\{i,j\}: i\in V(G), \, j\in V(G')\}.
\] 

\begin{Lemma}
	\label{d=2n-3}
	Let $n,t$ be integers with $n\geq 5$ and $t=4,\ldots,n$. Then there is no finite connected graph on $n$ vertices for which $\depth(S/J_G)=t$ and $\dim(S/J_G)=2n-3$. 
\end{Lemma}

\begin{proof}
	Suppose on the contrary that there exists a finite connected graph $G$ on $[n]$ with $\depth(S/J_G)=t$ and $\dim(S/J_G)=2n-3$. Since $\dim(S/J_G)=2n-3$, it follows from \eqref{dimension formula} that there is $T \subset [n]$ with $n-|T|+c_G(T)=2n-3$, and hence 
	\begin{equation}\label{useful equality}
	 n+|T|=c_G(T)+3.
	\end{equation}
	Since $n\geq 5$, one has $2n-3>n+1$ which implies that $|T|\geq 1$. Since $c_G(T)\leq n-1$, it follows that $c_G(T)+3\leq n+2$, and hence $|T|\leq 2$ by \eqref{useful equality}. 
	
	If $|T|=2$, then \eqref{useful equality} implies that $c_G(T)=n-1$, a contradiction. Let $|T|=1$ and $T=\{i\}$. Then $c_G(T)=n-2$ by \eqref{useful equality}. Thus, $G-i$ is the disjoint union of a complete graph $K_2$ and $n-3$ isolated vertices. Therefore, $G$ is either the join of $G-i$ and $K_1$ or obtained by attaching one leaf to one of the leaves of a star graph $K_{1,n-2}$. In both cases, $G$ is a block graph and $\depth(S/J_G)=n+1$ (Theorem~\ref{block}), a contradiction.
    \hspace{10.8cm}
\end{proof}

It follows from Theorem~\ref{maximum depth}, Lemma~\ref{d=2n-2} and Lemma~\ref{d=2n-3} that

\begin{Corollary}\label{higher d}
  Let $t,n,d$ be integers with $4\leq t\leq n+1$.  Suppose that one of the following conditions is satisfied:
  \begin{enumerate}
  	\item $n\geq 3$ and $d=2n-2$;
  	\item $n\geq 5$ and $d=2n-3$.
  \end{enumerate} 
Then there exists a finite connected  graph with $n$ vertices with $\depth(S/J_G)=t$ and $\dim(S/J_G)=d$ if and only if $t=n+1$. 
\end{Corollary}

\section{Two smallest values of $t$ in Question~\ref{Question}}\label{small depth section}

We now focus on two smallest values of $t$, namely $4$ and $5$. In \cite{RSK1}, binomial edge ideals whose quotient have depth equal to~$4$ are characterized (\cite[Theorem~5.3]{RSK1}).  

Let $G^c$ be the complementary graph \cite[p.~153]{HHgtm260} of $G$.

\begin{Theorem}[\cite{RSK1}]
    \label{depth=4}
	Let $G$ be a finite graph on $n\geq 4$ vertices. Then the following statements are equivalent:
	\begin{itemize}
		\item[(i)] $\depth(S/J_G) = 4$;
		\item[(ii)] $G =G'*K^c_2$ for some finite graph $G'$.
	\end{itemize}
\end{Theorem} 

\begin{Example}\label{n=4}
{\em 
	Let $G=K_2*K^c_2$.  One has $\depth(S/J_G) = 4$ (Theorem~\ref{depth=4}).  Let $V(K_2) = \{1,2\}$.  One has $\mathcal{C}(G)=\{\emptyset, \{1,2\}\}$ and, by \eqref{dimension formula}, $\dim(S/J_G)=5$.  
    }
\end{Example}

Now, by using Corollary~\ref{higher d} and Example~\ref{n=4}, one can assume that $n,t,d$ are integers with $n\geq 5$, $4\leq t\leq n$ and $n+1\leq d\leq 2n-4$.  

\begin{Theorem}\label{t=4}
	Let $n,d$ be integers with $n\geq 5$ and $n+1\leq d\leq 2n-4$. Then there exists a finite connected graph $G$ on $n$ vertices with $\depth(S/J_G)=4$ and $\dim(S/J_G)=d$.   
\end{Theorem}

\begin{proof}
  Let $\ell=1,\ldots,n-4$ and $d=n+\ell$. Let $H$ be the disjoint union of $K^c_{\ell+1}$ and $K_{n-\ell-3}$.  Set $G:=K^c_2*H$.  It then follows from Theorem~\ref{depth=4} that $\depth(S/J_G)=4$. One has $\mathcal{C}(G)=\{\emptyset, V(K^c_2), V(H)\}$ and hence 
  \[
  \dim(S/J_G)=\max \{n+1, n+\ell, 4\}=n+\ell=d,
  \]
as desired. 
\hspace{11.8cm}
\end{proof}

In \cite[Definition 3]{RSK2}, for the purpose of classifying all binomial edge ideals whose quotient has depth equal to~$5$, a class of graphs are introduced. In the sequel this class of graphs will be helpful for us.  
We say that a vertex $i$ of $G$ is {\em adjacent} to a vertex $j$ of $G$ if $\{i,j\} \in E(G)$.  A {\em neighbor} of a vertex $i$ of $G$ is a vertex $j$ of $G$ with $\{i,j\} \in E(G)$. Let $N_G(i)$ denote the set of neighbors of a vertex $i$ in $G$. The \emph{degree} of a vertex~$i$ in $G$ is defined to be $|N_G(i)|$.  
Let $W \subset [n]$ with $|W|=n-2$. 
Let $\mathcal{G}_W$ denote the set of graphs $G$ on $[n]$ for which there exist two nonadjacent vertices $i$ and $j$ of $G$ with $[n]-T=\{i,j\}$ and three disjoint subsets of $W$, say $V_0$, $V_1, V_2$, with $V_1 \neq \emptyset,V_2\neq \emptyset$ and $W=V_0\cup V_1\cup V_1$ for which the following conditions hold:
\begin{enumerate}
  \item $N_G(i)=V_0\cup V_1$, $N_G(j) = V_0 \cup V_2$;
  \item Any vertex in $V_1$ is adjacent to any vertex in $V_2$. 
\end{enumerate} 

\begin{Theorem}[\cite{RSK2}]
\label{depth=5}
  Let $n\geq 5$ be an integer and $W\subset [n]$ with $|W|=n-2$. If $G\in \mathcal{G}_W$ and $G\neq G'*K^c_2$ for any graph $G'$, then $\depth(S/J_G)=5$.	
\end{Theorem}
 
See \cite[Theorem~5]{RSK2} for a proof of Theorem \ref{depth=5}. 


\begin{Theorem}\label{t=5}
	Let $n,d$ be integers with $n\geq 5$ and $n+1\leq d\leq 2n-4$. Then there exists a finit connected graph $G$ on $n$ vertices with $\depth(S/J_G)=5$ and $\dim(S/J_G)=d$.   
\end{Theorem}

\begin{proof}
	Let $\ell=1,\ldots,n-4$ and $d=n+\ell$. Let $W=[n]-\{1,2\}$.  Let $G$ be a graph belonging to $\mathcal{G}_W$ with $V_0=\emptyset$ and $|V_1|=1$ for which $G|_{V_2}$ is the disjoint union of $K^c_{\ell}$ and $K_{n-\ell-3}$, where $N_G(1)=V_1$ and $N_G(2)=V_2$.  Observe that $G \neq G'*K^c_2$ for any graph $G'$, since the only vertex of degree~$n-2$ in $G$ is the vertex in $V_1$. One has $\depth(S/J_G)=5$ (Theorem~\ref{depth=5}).
	Furthermore, one has $\mathcal{C}(G)=\{\emptyset, V_1, V_1\cup \{2\}, V_2\}$
	and hence 
	\[
	\dim(S/J_{G})=\max \{n+1, n+\ell, 5\}=n+\ell = d,
	\]
as desired. 
\hspace{11.8cm}
\end{proof}

\section{The remaining values of $t$ in Question~\ref{Question} from $6$ to $n$}\label{other values section}

By virtue of Corollary~\ref{higher d}, Example~\ref{n=4} and Theorems~\ref{t=4} and~\ref{t=5}, to complete the answer to Question~\ref{Question}, it remains to discuss integers $n,t,d$ with $6\leq t\leq n$ and $n+1\leq d\leq 2n-4$.  We construct a family of graphs which plays an important role for our goal. Let $n,t,\ell,p$ be integers with $6\leq t\leq n$, $1\leq \ell \leq t-4$ and $2\leq p\leq n-t+2$. 

We consider the following procedure to construct the graph $G^t_{\ell,p}$ with $n$ vertices: 
\begin{itemize}
	\item Let $\tilde{H}_{t,\ell}$ be the disjoint union of $K^c_{\ell-1}$ and $K_{t-\ell-3}$. 
	\item Set $H(t,\ell):=K_1*\tilde{H}_{t,\ell}$, where $V(K_1)=\{q\}$.
	\item Let $\tilde{L}_{t,p}$ be the disjoint union of $K^c_{p-1}$ and $K_{n-t-p+3}$. 
	\item Set $L(t,p):=K^c_2*\tilde{L}_{t,p}$, where $V(K^c_2)=\{i,j\}$. 
	\item Finally, we define $G^t_{\ell,p}$ to be the graph obtained from $H(t,\ell)$ and $L(t,p)$ (on disjoint sets of vertices) by identifying the vertices $i$ and $q$. 
\end{itemize}
  
We compute the Krull dimension and the depth of the binomial edge ideal of $G^t_{\ell,p}$. 

\begin{Theorem}\label{final dim}
	Let $n,t,\ell,p$ be integers with $$6\leq t\leq n, \, 1\leq \ell \leq t-4, \, 2\leq p\leq n-t+2.$$ Then one has 
	\[
	\dim (S/J_{G^t_{\ell,p}})=n+\ell+p-2.
	\] 
	In particular, given integers $n,d$ with $n\geq 6$ and $n+1\leq d\leq 2n-4$, there exists a finite connected graph $G$ on $n$ vertices with $\dim (S/J_G)=d$.  
\end{Theorem}

\begin{proof}
	First, observe that 
	\[
	\mathcal{C}(G^t_{\ell,p})=\{\emptyset, \{i\}, \{i,j\}, V(\tilde{L}_{t,p})\}.
\]
It then follows from \eqref{dimension formula} that 
	\[
	\dim (S/J_{G^t_{\ell,p}})=\max \big{\{}n+1, n+\ell, n+\ell+p-2, t\big{\}}=n+\ell+p-2,
	\] 
	since $\ell\geq 1$, $p\geq 2$ and $n\geq t$. 
	In particular, one has 
	\[
	\dim (S/J_{G^t_{1,p}})=n+p-1
	\] 
	for $p=2,\ldots, n-t+1$ and 
	\[
	\dim (S/J_{G^t_{\ell,n-t+2}})=2n-t+\ell
	\] 
	for $\ell=1,\ldots, t-4$. This in particular implies that, given integers $n,d$ with $n\geq 6$ and $n+1\leq d\leq 2n-4$, there exists a finite connected graph $G$ on $n$ vertices with $\dim (S/J_G)=d$, as desired.
    \hspace{8.7cm}
\end{proof}

In \cite{KS}, the class of \emph{generalized block graphs} is defined as a generalization of the class of block graphs.  Let $G$ be a chordal graph with the property that for every three maximal cliques of $G$ whose sets of vertices have a nonempty intersection, the intersection of the vertex sets of each pair of them is same (\cite[Figure~1 and Figure~2]{KS}). Every block graph is also a generalized block graph. An explicit formula for $\depth(S/J_G)$ is known for a generalized block graph $G$ (\cite[Theorem~3.2]{KS}).  A \emph{cut set} of a finite graph $G$ is a subset $C$ of $V(G)$ for which the induced subgraph on $V(G) \setminus C$ is disconnected.  Furthermore, a \emph{minimal cut set} of $G$ is a cut set which is minimal under inclusion.  The \emph{clique number} of $G$ is the maximum size $\omega(G)$ of the cliques of $G$.  Let $G$ be a generalized block graph on $[n]$. For  
$i = 1,\ldots,\omega(G)-1$, set 
\[
\mathcal{A}_i(G) = \{A \subseteq [n] : |A| = i, A~\textit{is~a~minimal~cut~set~of}~G\}.
\]

\begin{Theorem}[\cite{KS}]
\label{generalized block}
	Let $G$ be a finite connected generalized block graph on $[n]$. Then
	\[
	\depth(S/J_G) = n+1-\sum_{i=2}^{\omega(G)-1} (i-1)|\mathcal{A}_i(G)|.
	\] 
\end{Theorem}

Let $G$ be a finite graph on $[n]$ and $i \in V(G)$. 
Let $G_i$ denote the finite graph on $[n]$ whose edge set is 
\[
E(G)\cup \big{\{}\{a,b\}: \{a,b\}\subseteq N_G(i)\big{\}}.
\]

The following classical result is helpful to compute the depth of $S/J_{G^t_{\ell,p}}$. 

\begin{Lemma}\label{depth lemma}
	Let $R$ be a polynomial ring over a field $K$.  Let $M$, $N$ and $P$ be finitely generated graded $R$-modules. If 
	\[
	0\rightarrow M\rightarrow N\rightarrow P\rightarrow 0
	\]
	is a short exact sequence and if $\depth(N) > \depth(P)$, then
	\[
	\depth(M)=\depth(P)+1.
	\]
\end{Lemma}


We are ready to compute the depth of the binomial edge ideal of $G^t_{\ell,p}$. 

\begin{Theorem}\label{final depth}
  Let $n,t,\ell,p$ be integers with $$6\leq t\leq n, \, 1\leq \ell \leq t-4, \, 2\leq p\leq n-t+2.$$ Then one has $\depth (S/J_{G^t_{\ell,p}})=t$.
\end{Theorem}

\begin{proof}
  Set $G:=G^t_{\ell,p}$.  It follows from \cite[Lemma~4.8]{O} that $$J_G=J_{G_i}\cap ((x_i,y_i)+J_{G-i})$$ and hence one has the short exact sequence  
  \[
  0 \rightarrow S/J_G \rightarrow S/J_{G_i}\oplus S_i/J_{G-i} \rightarrow S_i/J_{G_i-i} \rightarrow 0,
  \]  	
  where $S_i$ is the polynomial ring in the same variables as $S$ except $x_i$ and $y_i$. 
  
  First, observe that $G_i$ is the graph consisting of two maximal cliques which have exactly $n-t+2$ common vertices. In particular, $G_i$ is a generalized block graph and hence one has $\depth(S/J_{G_i})=n+1-(n-t+1)=t$ (Theorem~\ref{generalized block}). 
  
  It is clear that $G-i$ is just the disjoint union of $H^{(1)}=K^c_{\ell-1}$, 
  $H^{(2)}=K_{t-\ell-3}$ and $H^{(3)}=L(t,p)-i$. Therefore, one has 
  $\depth(S_i/J_{G-i})= \sum_{s=1}^3 \depth(R_s/J_{H^{(s)}})$, where $R_s$ is the polynomial ring over $K$ in the variables correspond to the vertices of $H^{(s)}$. On the other hand, since $L(t,p)-i$ is a block graph, it follows from Theorem~\ref{block} that
  \begin{eqnarray*}
      &\depth(R_1/J_{H^{(1)}})=2\ell-2,\, \, \,  
      \depth(R_2/J_{H^{(2)}})=t-\ell-2,& \\
      &\depth(R_3/J_{H^{(3)}})=n-t+4.& 
  \end{eqnarray*}
    It then follows that $\depth(S_i/J_{G-i})=n+\ell$. 
  
Since 
  \[
  \depth(S/J_{G_i}\oplus S_i/J_{G-i})=
  \min\{\depth(S/J_{G_i}), \depth(S_i/J_{G-i})\}
      =\min\{t, n+\ell\}
  \] 
and since $t\leq n$ and $\ell\geq 1$, one has 
  \begin{equation}\label{depth N}
   \depth(S/J_{G_i}\oplus S_i/J_{G-i})=t.
  \end{equation}
  
  Finally, observe that $G_i-i$ is a generalized block graph on $n-1$ vertices whose unique minimal cut set is of cardinality~$n-t+2$, and it follows from Theorem~\ref{generalized block} that $\depth(S_i/J_{G_i-i})=(n-1)+1-(n-t+1)=t-1$. It then follows from \eqref{depth N} and Lemma~\ref{depth lemma} that $\depth(S/J_G)=t$, as desired. 
  \hspace{5.5cm}
\end{proof}

Now, 
combining Theorem~\ref{final dim} and Theorem~\ref{final depth}, it follows that

\begin{Corollary}\label{remaining cases}
	Given integers $t,n,d$ with $6\leq t\leq n$ and $n+1\leq d\leq 2n-4$, there exists a finite connected graph $G$ on $n$ vertices with $\depth(S/J_G)=t$ and $\dim(S/J_G)=d$.
\end{Corollary}



\end{document}